\documentclass{amsart}

\usepackage{amsmath,amsthm,amssymb,amsfonts}
\input xy
\xyoption{all}
\usepackage{enumerate}
\usepackage{hyperref}

\makeatletter
\newcommand{\labitem}[2]{%
\def\@itemlabel{\textbf{#1}}
\item
\def\@currentlabel{#1}\label{#2}}
\makeatother

\theoremstyle{plain}
\newtheorem{thm}{Theorem}[section]
\newtheorem{prp}[thm]{Proposition}
\newtheorem{dfn-prp}[thm]{Definition-Proposition}
\newtheorem{lem}[thm]{Lemma}
\newtheorem{cor}[thm]{Corollary}
\newtheorem*{cnj}{Conjecture}
\newtheorem*{thmA}{Theorem A}
\newtheorem*{thmB}{Theorem B}
\newtheorem*{thmC}{Theorem C}

\theoremstyle{definition}
\newtheorem{dfn}[thm]{Definition}

\theoremstyle{remark}

\newtheorem{exm}[thm]{Example}

\newcommand{\Hom}{\operatorname{Hom}}
\newcommand{\Mor}{\operatorname{Mor}}

\newcommand{\Inn}{\operatorname{Inn}}
\newcommand{\Out}{\operatorname{Out}}

\newcommand{\Id}{\mathrm{Id}}

\newcommand{\into}{\hookrightarrow}

\newcommand{\from}{\leftarrow}

\newcommand{\SL}{\operatorname{SL}}

\newcommand{\res}{\operatorname{r}}
\newcommand{\resm}{\operatorname{res}}

\newcommand{\tr}{\operatorname{t}}
\newcommand{\trm}{\operatorname{tr}}
\newcommand{\iso}{\operatorname{iso}}

\newcommand{\Syl}{\operatorname{Syl}}

\newcommand{\Ob}{\mathrm{Ob}}

\newcommand{\Top}{\mathsf{Top}}
\newcommand{\Mack}{\mathsf{Mack}}	

\newcommand{\hocolim}{\operatornamewithlimits{hocolim}}

\newcommand{\la}{\langle}
\newcommand{\ra}{\rangle}

\newcommand{\ol}{\overline}

\makeatletter
\@tempcnta=\@ne
\def\@nameedef#1{\expandafter\edef\csname #1\endcsname}
\loop\ifnum\@tempcnta<27
  \@nameedef{C\@Alph\@tempcnta}{\noexpand\mathcal{\@Alph\@tempcnta}}
  \advance\@tempcnta\@ne
\repeat

\makeatletter
\@tempcnta=\@ne
\def\@nameedef#1{\expandafter\edef\csname #1\endcsname}
\loop\ifnum\@tempcnta<27
  \@nameedef{B\@Alph\@tempcnta}{\noexpand\mathbb{\@Alph\@tempcnta}}
  \advance\@tempcnta\@ne
\repeat

\makeatletter
\@tempcnta=\@ne
\def\@nameedef#1{\expandafter\edef\csname #1\endcsname}
\loop\ifnum\@tempcnta<27
  \@nameedef{F\@Alph\@tempcnta}{\noexpand\mathsf{\@Alph\@tempcnta}}
  \advance\@tempcnta\@ne
\repeat

\def\O{\ensuremath{\mathcal{O}}}
\newcommand{\Fun}{\mathsf{Fun}}
\newcommand{\fgmod}[1]{#1\text{-}\mathsf{mod}}
\newcommand{\Mod}[1]{#1\text{-}\mathsf{Mod}}

\newcommand{\Chi}{\mathcal{X}} 

\begin{document}
\title{Mackey functors and sharpness for fusion systems}
\author{Antonio D\'iaz}
\email{adiaz@agt.cie.uma.es}
\address{Departamento de \'Algebra, Geometr\'ia y Topolog\'ia, 
	Universidad de M\'alaga, 
	Apdo correos 59, 
	29080 M\'alaga, 
	Spain}
\author{Sejong Park}
\email{sejong.park@nuigalway.ie}
\address{School of Mathematics, Statistics and Applied Mathematics, 
	National University of Ireland, Galway, 
	University Road, Galway,
	Ireland}
\thanks{\rm The first author is supported by MICINN grant RYC-2010-05663 and partially supported by FEDER-MCI grant MTM2010-18089 and Junta de Andaluc{\'\i}a grant FQM-213.}



\begin{abstract}
We develop the fundamentals of Mackey functors in the setup of fusion systems including an acyclicity condition as well as a parametrization and an explicit description of simple Mackey functors. Using this machinery we extend Dwyer's sharpness results  to exotic fusion systems $\CF$ on a finite $p$-group $S$ with an abelian subgroup of index $p$.
\end{abstract}


\maketitle

\section{Introduction.}\label{section:Introduction}

In \cite{BLO2003theory}, Broto, Levi and Oliver introduced a family of topological spaces known as $p$-local finite groups, based on fusion systems defined by Puig. These objects behave homotopically like the $p$-completed classifying spaces of finite groups but without necessarily being one of them. In particular, these spaces may be described in different ways as a homotopy colimit of simpler spaces; see the subgroup and centralizer homology decompositions in \cite[\S 2]{BLO2003theory}.

Let $\CF$ be a saturated fusion system on the finite $p$-group $S$. By the groundbreaking recent work of Chermak \cite{C2011}, we can associate to each saturated fusion system $\CF$ a unique $p$-local finite group that plays the role of the classifying space $B\CF$ of $\CF$. The subgroup homology decomposition for $B\CF$ permits  to reconstruct this classifying space by gluing together classifying spaces $BP$, where $P$ runs over the collection of  subgroups of $S$ that are $\CF$-centric.

For the finite group case, i.e., for $\CF=\CF_S(G)$ with $G$ a finite group containing $S$ as a Sylow $p$-subgroup, the subgroup homology decomposition possesses one further feature called \emph{sharpness}. To explain this terminology we need first to consider the orbit category $\CO(\CF)$ and the centric orbit category $\CO(\CF^c)$. Objects in $\CO(\CF)$ are all subgroups of $S$ and morphisms are given by $\Hom_{\CO(\CF)}(P,Q)=\Inn(Q)\backslash\Hom_{\CF}(P,Q)$, where $\Inn(Q)$ acts by post-composition. The centric orbit category $\CO(\CF^c)$ is the full subcategory of $\CO(\CF)$ with objects the $\CF$-centric subgroups of $S$. Then the subgroup decomposition reads as 
\[
B\CF\simeq \hocolim_{\CO(\CF^c)} \widetilde{B}, 
\]
where $\widetilde{B}\colon\CO(\CF^c)\rightarrow \Top$ is a functor such that $\widetilde B(P)$ has the homotopy type of $BP$ for each $\CF$-centric subgroup $P\leq S$.  
By the work of Bousfield and Kan \cite{BK1972}, the above homotopy colimit gives rise to a first quadrant cohomology spectral sequence
\[
E^{i,j}_2={{\varprojlim}^i_{\CO(\CF^c)}}  H^j(-;\BF_p)\Rightarrow H^{i+j}(B\CF;\BF_p).
\]
Sharpness means that this spectral sequence collapses onto the vertical axis. This is equivalent to that $H^j(-;\BF_p)$ is \emph{acyclic} as a contravariant functor over $\CO(\CF^c)$ for any $j\geq 0$, which means that the higher limits  ${{\varprojlim}^i_{\CO(\CF^c)}}  H^j(-;\BF_p)$ vanish for all $i\geq 1$ and $j\geq 0$. It was proven for the finite group case $\CF=\CF_S(G)$ by Dwyer in \cite[Theorem 10.3]{D1998}, where it is referred to as \emph{subgroup-sharpness for the collection of $p$-centric subgroups}. More precisely, Dwyer's result states that the higher limits ${{\varprojlim}^i_{\CO^c_p(G)}} H^j(-;\BF_p)$ vanish, where $\CO^c_p(G)$ is the category with objects the $p$-centric subgroups $P$ of $G$ and morphisms $\Mor_{\CO^c_p(G)}(P,Q)=Q\backslash N_G(P,Q)$, where $N_G(P,Q) = \{ x\in G \mid {}^{x}P \leq Q \}$. But in fact, by \cite[Lemma
1.3]{{BLO2003equiv}}, the higher limits over $\CO(\CF^c)$ and $\CO^c_p(G)$ coincide (see also the proof of Theorem~\hyperref[theorem:group-case]{B} below).


An important feature of the cohomology functor $H^j(-;\BF_p)$ is that it is a global Mackey functor (see \cite[\S 1]{W1993} for a definition). Roughly speaking, this means that transfer and restriction for cohomology satisfy various compatibility conditions including the Mackey decomposition formula.



Given a fusion system $\CF$ on a finite $p$-group $S$, we can define a \emph{Mackey functor} for $\CF$ (over a commutative ring $k$) analogously (see Section \ref{section:Mackey functors.} for a precise definition). It consists of a function $M$ which assigns a $k$-module $M(P)$ to each subgroup $P$ of $S$, homomorphisms $\res^P_Q \colon M(P) \to M(Q)$ and $\tr^P_Q \colon M(Q) \to M(P)$ for $Q\leq P\leq S$ and isomorphisms $\iso(\varphi) \colon M(P) \to M(\varphi(P))$ for any isomorphism $\varphi$ of $\CF$ and they satisfy (among others) the Mackey decomposition formula: 
\begin{equation}
\label{equ:MackeydecompO(F)intro}
\res^{P}_{Q} \circ \tr^{P}_{R} = \sum_{x \in [Q \backslash P/R]} \tr^{Q}_{Q\cap{}^{x}R} \circ \res^{{}^{x}R}_{Q\cap {}^{x}R} \circ \iso(c_x|_R)\colon M(R)\to M(Q),
\end{equation}
where $Q,R\leq P\leq S$ and $[Q\backslash P/R]$ denotes a set of representatives of the double cosets $QxR$ in $P$. We shall also introduce the notion of \emph{$\CX$-restricted Mackey functor} for $\CF$, where $\CX$ is an $\CF$-overconjugacy closed collection of subgroups of $S$, for instance, the collection $\CF^c$ of the $\CF$-centric sugroups of $S$. In this case, the Mackey decomposition formula for $\CX$ takes the form
\begin{equation}
\label{equ:MackeydecompO(X)intro}
\res^{P}_{Q} \circ \tr^{P}_{R} = \sum_{x \in [Q \backslash P/R]_{\CX}} \tr^{Q}_{Q\cap{}^{x}R} \circ \res^{{}^{x}R}_{Q\cap {}^{x}R} \circ \iso(c_x|_R)\colon M(R)\to M(Q),
\end{equation}
where $Q,R\leq P$ are subgroups in $\CX$ and $[Q \backslash P/R]_{\CX}$ is a set of representatives of the double cosets $QxR$ in $P$ with $Q \cap {}^{x}R \in \CX$ . 

The notions of Mackey functors for $\CF$ and $\CX$-restricted Mackey functor for $\CF$ may also be stated in terms of pairs of functors $M=(M^*, M_*)\colon\CO(\CF)\to\Mod{k}$ and $M=(M^*, M_*)\colon\CO(\CX)\to\Mod{k}$, respectively, where $M^*$ is contravariant, $M_*$ is covariant, $\CO(\CX)$ is the full subcategory of $\CO(\CF)$ with object set $\CX$, and $\Mod{k}$ is the category of $k$-modules. In our main applications, we only consider Mackey functors where the ground ring $k$ is a field and whose values are in the category $\fgmod{k}$ of finite dimensional $k$-vector spaces. Roughly speaking, the contravariant part $M^*$ is obtained from $\res$ and $\iso$, and the covariant part $M_*$  from $\tr$ and $\iso$. This equivalence and the full detailed definitions of Mackey functors for fusion systems are given in Section \ref{section:Mackey functors.}. 

The acyclicity of Mackey functors over suitable categories is well-known and traces back to the paper \cite{JM1992} of Jackowski and McClure. For the setup of fusion systems, the combination of  Jackowski and McClure's result with the work of D\'iaz and Libman \cite{DL2009} gives the following acyclicity result.

\begin{thmA} 
Let $\CF$ be a  saturated fusion system on a finite $p$-group $S$.  Then for every $\CF^c$-restricted Mackey functor $M = (M^*,M_*) \colon\CO(\CF^c)\to\Mod{\BZ_{(p)}}$ for $\CF$, the contravariant part $M^*$ is acyclic: $\varprojlim^i_{\CO(\CF^c)} M^*|_{\CO(\CF^c)}=0$ for $i>0$.
\end{thmA}

This theorem is proven in Section~\ref{S:acyclic-Mackey} and it will play a key role in the proof of Theorem C below. Motivated by Dwyer's sharpness result and the above theorem, we make the following conjecture,  which we name after the original notion of sharpness.


\begin{cnj}[Sharpness for fusion systems]
\label{theorem:conjectureA}
Let $\CF$ be a saturated fusion system on a finite $p$-group and let $M=(M^*,M_*)\colon \CO(\CF)\to \fgmod{\BF_p}$ be a Mackey functor over the full orbit category of $\CF$. Then the higher limits of the contravariant part $M^*$ over the centric orbit category vanish: $\varprojlim^i_{\CO(\CF^c)} M^*|_{\CO(\CF^c)}=0$ for $i>0$.
\end{cnj}

This conjecture includes as a particular case the acyclicity of the cohomology functor $H^j(-;\BF_p)\colon\CO(\CF^c)\rightarrow \fgmod{\BF_p}$ for any $j\geq 0$, which is one of the open problems listed in \cite[III.\S7]{AKO2011}.  Dwyer's subgroup-sharpness for the collection of $p$-centric subgroups shows that the conjecture holds for $M$ equal to the cohomology functor $H^j(-;\BF_p)\colon\CO(\CF)\rightarrow \fgmod{\BF_p}$ and $\CF=\CF_S(G)$ the fusion system of a finite group $G$ with Sylow $p$-subgroup $S$.


In this paper, we confirm the sharpness conjecture for some special cases. First, the conjecture holds when $\CF$ is a group fusion system. It is basically the result of Jackowski--McClure (see Proposition \ref{T:JM-acyclicity}) translated to our setting, using two lemmas that we can either mod out by $p'$-subgroups (\cite[Lemma
1.3]{{BLO2003equiv}}) or discard non $p$-centric subgroups (\cite[Theorem 1.2]{G2002}) without changing the higher limits.

\begin{thmB}
Let $\CF=\CF_S(G)$ be the fusion system of the finite group $G$ with $S\in \Syl_p(G)$ and let $M=(M^*,M_*)\colon \CO(\CF)\to \fgmod{\BF_p}$ be a Mackey functor over the full orbit category of $\CF$. Then the higher limits of $M^*$ over the orbit centric category vanish: $\varprojlim^i_{\CO(\CF^c)} M^*|_{\CO(\CF^c)}=0$ for $i>0$.
\end{thmB}

For exotic fusion systems the strategy used to prove Theorem B can not be applied because of the lack of an ambient finite group inducing the given fusion. Nevertheless, the development of some standard machinery for Mackey functors provides us with a suitable setting to study the sharpness conjecture even in the exotic case. These tools consist of a parametrization and a concrete description of simple Mackey functors (Section \ref{section:Simple Mackey functors}) as well as the aforementioned acyclicity result for $\CF^c$-restricted Mackey functors (Theorem A above and Section \ref{S:acyclic-Mackey}). In short, we filter the given Mackey functor $M$ over the full orbit category $\CO(\CF)$ with simple subquotients, and use the concrete description of those simple subquotients to show their restriction to $\CO(\CF^c)$ are $\CF^c$-restricted Mackey functors. With this method, we obtain the following theorem.

\begin{thmC}
Let $S$ be a nonabelian finite $p$-group with an abelian subgroup of index $p$. For any  saturated fusion system $\CF$ on $S$ and any Mackey functor 
\[
M=(M^*,M_*)\colon \CO(\CF)\rightarrow \fgmod{\BF_p},
\]
we have 
\[
{{\varprojlim}^i_{\CO(\CF^c)}} M^*|_{\CO(\CF^c)}= 0\qquad\forall i>0.
\]
\end{thmC}

We consider the above case because such $p$-groups $S$ afford plenty of exotic fusion systems, even though $S$ is close to being abelian. Thus Theorem C with $M=H^j(-;\BF_p)$ extends Dwyer's result to those exotic fusion systems. All  fusion systems on $p$-groups of $p$-rank two for $p$ odd are classified by the first author and Ruiz and Viruel \cite{RV2004, DRV2007}, and all exotic fusion systems among them are afforded by $p$-groups $S$ having an abelian subgroup $A$ of index $p$. More recently Oliver \cite{Oindp1} set out to classify all simple  fusion systems on such $p$-groups $S$. He has completed the case where $p$ is odd and $A$ is not $\CF$-essential, and found that most of them are exotic. We hasten to point out that Theorem C does not require the full classification of such fusion systems, though. In particular, it applies to all finite $p$-groups $S$ having an abelian subgroup $A$ of index $p$, and does not require the assumption that $A$ is not $\CF$-essential. Theorems B and C are proven in Section \ref{section:Examples.}.

\textbf{Acknowledgements:} We are grateful to Adam Glesser and Radu Stancu for their support and helpful comments in several occasions during the long period of gestation of this work. We are also in debt to \'Ecole Polythechnique F\'ed\'erale de Lausanne, Universit\'e de Picardie and Universidad de M\'alaga for their hospitality and support.

\section{Mackey functors for fusion systems}\label{section:Mackey functors.}

Let $\CF$ be a fusion system on a finite $p$-group $S$. Let $\CX$ be a collection of  subgroups of $S$ which is {\em closed under $\CF$-overconjugacy}, in the sense that if $\varphi\colon P\to Q$ is a morphism in $\CF$ and $P$ belongs to $\CX$, then so does $Q$. For example, the collection $\CF^c$ of the $\CF$-centric subgroups of $S$ is closed under $\CF$-overconjugacy. Let $\CO(\CX)$ denote the full subcategory of the orbit category $\CO(\CF)$ with object set $\CX$. For a morphism $\varphi\colon P\to Q$ in $\CF$, let $[\varphi]$ denote the $\Inn(Q)$-orbit of $\varphi$. For subgroups $Q\leq P\leq S$, let $\iota_Q^P\colon Q\to P$ denote the inclusion map. For a subgroup $P$ of $S$ and subgroups $Q$, $R$ of $P$, let $[Q\backslash P/R]$ denote a set of representatives of the double cosets $QxR$ in $P$, and let 
\[
	[Q\backslash P/R]_{\CX} = \{ x \in [Q\backslash P/R] \mid Q \cap {}^{x}R \in \CX \}.
\]

\begin{dfn} \label{D:Mackey-bivariant}
Let $\CF$ be a fusion system and let $\CX$ be an $\CF$-overconjugacy closed collection of subgroups of $S$. An  {\em $\CX$-restricted Mackey functor} for $\CF$ over a commutative ring $k$ with identity element is a pair of functors 
\[
	M=(M^*, M_*)\colon\CO(\CX)\to\Mod{k}
\]
satisfying the following conditions.
\begin{enumerate}
\item (Bivariance) $M^*\colon \CO(\CX) \to \Mod{k}$ is a contravariant functor, $M_*\colon \CO(\CX) \to \Mod{k}$ is a covariant functor, and $M^*(P) = M_*(P) =: M(P)$ for all objects~$P$ of $\CX$.
\item (Isomorphism) $M^*([\alpha]) =  M_*([\alpha]^{-1})$ if $[\alpha]$ is an isomorphism in $\CO(\CX)$.
\item (Mackey decomposition) If $Q, R \leq P \leq S$ are in $\CX$, then
\[
	M^*([\iota^P_Q]) \circ M_*([\iota^P_R]) = \sum_{x\in [Q\backslash P/R]_{\CX}} M_*([\iota^Q_{Q\cap{}^{x}R}]) \circ M^*([\iota^{{}^{x}R}_{Q\cap{}^{x}R}]) \circ M_*([c_x|_R]).
\]
\end{enumerate}
When $\CX$ consists of all subgroups of $S$, we simply say that $M$ is a {\em Mackey functor} for $\CF$.
\end{dfn}

In the above definition, the Mackey decomposition condition is in fact ``$\CX$-truncated Mackey decomposition'' in which we take only those double coset representatives whose corresponding intersections belong to $\CX$. The reason why we consider these restricted Mackey functors is that the $\CF^c$-restricted Mackey functors for $\CF$ have a nice acylicity property (Theorem~\hyperref[T:Mackey-acyclic-over-OFc]{A}) which is crucial for proving Theorem C.

With a slight change of notation, we can reformulate the definition of Mackey functors in a way that is closer to the classical definition as follows.

\begin{prp}\label{D:Mackey-centric-res-tr-iso}
Let $\CF$ be a fusion system on a finite $p$-groups $S$ and let $\CX$ be an $\CF$-overconjugacy closed collection of subgroups of $S$. An  {\em $\CX$-restricted Mackey functor} for $\CF$ over a commutative ring $k$ with identity element can be prescribed as a function
\[
	M\colon \Ob(\Chi) \to \Mod{k} 
\]
together with $k$-module homomorphisms
\begin{gather*}
	\res^P_Q \colon M(P) \to M(Q) \\
	\tr^P_Q \colon M(Q) \to M(P) \\
	\iso(\varphi) \colon M(P) \to M(\varphi(P))
\end{gather*}
for all objects $Q \leq P$ in $\Chi$ and for all isomorphisms $\varphi\colon P\to \varphi(P)$ in $\CF$ satisfying the following conditions.
\begin{enumerate}
\item (Identity) \label{D:Mackey-centric-res-tr-iso:Identity}$\res^P_P = \tr^P_P = \iso(c_x|_P) = \Id_{M(P)}$ for $P$ in $\Chi$ and for $x\in P$.
\item (Transitivity) \label{D:Mackey-centric-res-tr-iso:Transitivity}$\res^Q_R \circ \res^P_Q = \res^P_R$, $\tr^P_Q \circ \tr^Q_R = \tr^P_R$, $\iso(\psi) \circ \iso(\varphi)= \iso(\psi\circ\varphi)$ for $R \leq Q \leq P$ in $\Chi$ and for isomorphisms $\varphi$, $\psi$ in $\Chi$ such that $\psi\circ\varphi$ is defined.
\item (Conjugation) \label{D:Mackey-centric-res-tr-iso:Conjugation}$\iso(\varphi|_Q) \circ \res^P_Q = \res^{\varphi(P)}_{\varphi(Q)} \circ \iso(\varphi)$, $\iso(\varphi) \circ \tr^P_Q = \tr^{\varphi(P)}_{\varphi(Q)} \circ \iso(\varphi|_Q)$ for $Q\leq P$ in $\Chi$ and for isomorphisms $\varphi\colon P\to \varphi(P)$ in $\Chi$.
\item (Mackey decomposition) \label{D:Mackey-centric-res-tr-iso:MackeyDecomposition} For $Q, R \leq P$ in $\Chi$,
\[
	\res^{P}_{Q} \circ \tr^{P}_{R} = \sum_{x \in [Q \backslash P/R]_{\CX}} \tr^{Q}_{Q\cap{}^{x}R} \circ \res^{{}^{x}R}_{Q\cap {}^{x}R} \circ \iso(c_x|_R).
\] 
\end{enumerate}
\end{prp}

\begin{proof}
If $M$ is an $\CX$-restricted Mackey functor for $\CF$, set
\begin{gather*}
	\res^P_Q = M^*([\iota^P_Q])\colon M(P) \to M(Q),\\
	\tr^P_Q = M_*([\iota^P_Q])\colon M(Q) \to M(P),\\
	\iso(\varphi) = M_*([\varphi]) = M^*([\varphi^{-1}])\colon M(P) \to M(\varphi(P)),
\end{gather*}
where $Q \leq P$ are in $\Chi$ and $\varphi\colon P\to \varphi(P)$ is an isomorphism in $\CF$. Conversely, if $M\colon \Ob(\Chi) \to \Mod{k}$ is a function satisfying the above four conditions, for a morphism $\varphi\colon P\to \varphi(P)$ of $\CF$ with $P$, $Q$ in $\CX$, set
\begin{gather*}
	M^*([\varphi]) = \iso(\ol{\varphi}^{-1}) \circ \res^Q_{\varphi(P)}, \\
	M_*([\varphi]) = \tr^Q_{\varphi(P)} \circ \iso(\ol{\varphi}), 
\end{gather*}
where $\ol{\varphi}\colon P \to \varphi(P)$ is the induced isomorphism. It is straightforward to check that the desired conditions are satisfied in both directions.
\end{proof}

The $\CX$-restricted Mackey functors for $\CF$ form a category where a morphism $\eta\colon M \to N$ is a collection of $k$-module homomorphisms $\{\eta_P\colon M(P)\to N(P)\}_{P\in\CX}$ such that for any $[\varphi]\in \Hom_{\CO(\CX)}(P,Q)$ we have $\eta_P\circ M^*([\varphi])=N^*([\varphi])\circ \eta_Q$ and 
$\eta_Q\circ M_*([\varphi])=N_*([\varphi])\circ \eta_P$. We denote by $\Mack^\CX_k(\CF)$ the category of $\CX$-restricted Mackey functors for $\CF$ over $k$, and we simply write $\Mack_k(\CF)$ for the category of Mackey functors for $\CF$ over $k$. We also denote by $\Fun^{\mathsf{bi}}(\CO(\CX),\Mod{k})$ the category of those pairs $(M^*,M_*)$ satisfying the Bivariance condition and the Isomorphism condition in Definition~\ref{D:Mackey-bivariant}.  Then $\Mack^\CX_k(\CF)$ embeds into $\Mack_k(\CF)$ by assigning zero modules for all $P\leq S$ not in $\CX$. Also, $\Mack^\CX_k(\CF)$ embeds into $\Fun^{\mathsf{bi}}(\CO(\CX),\Mod{k})$ by forgetting the Mackey decomposition condition and we have the following commutative diagram
\begin{equation}\label{equ:MackresandFunBires}
\xymatrix{
	\Fun^{\mathsf{bi}}(\CO(\CX),\Mod{k}) & \Fun^{\mathsf{bi}}(\CO(\CF),\Mod{k}) \ar[l]_{\resm}  \\
	\Mack^\CX_k(\CF) \ar@{^{(}->}[r] \ar@{^{(}->}[u] & \Mack_k(\CF).  \ar@{^{(}->}[u]
}
\end{equation}
For $M\in\Mack_k(\CF)$, its restriction $M|_{\CO(\CX)}$ to $\CO(\CX)$ is not necessarily a Mackey functor for $\CX$; it is Mackey if and only if it satisfies the Mackey decomposition condition for $\CX$.

For example, the $j$th cohomology functor $H^j(-;k)$ with trivial coefficients $k$ is a global Mackey functor in the sense of Webb~\cite{W1993} and its restriction $H^j(-,k)|_{\CO(\CF)}$ to $\CO(\CF)$ is a Mackey functor for $\CF$. But the further restriction $H^j(-,k)|_{\CO(\CX)}$ to $\CO(\CX)$ may not be an $\CX$-restricted Mackey functor for $\CF$ (see Example \ref{example:nonmackey}). A criterion which ensures that a Mackey functor $M$ for a  fusion system $\CF$ restricts to an $\CF^c$-restricted Mackey functor $M|_{\CO(\CF^c)}$ is given in Proposition~\ref{prop:MackeyFvsMackeyFc}. We will see examples of $\CF^c$-restricted Mackey functors for $\CF$ in the proof of Theorem~\hyperlink{theorem:index-p-case}{C}.

\section{Simple Mackey functors.}\label{section:Simple Mackey functors}

Let $\CF$ be a fusion system on the finite $p$-group $S$, let $\CX$ be an $\CF$-overconjugacy closed collection of subgroups of $S$ and let $k$ be a commutative ring with identity element. The category $\Mack^\CX_k(\CF)$ of $\CX$-restricted Mackey functors for $\CF$ over $k$ is an abelian category in which kernels and cokernels are constructed ``objectwise''. In this section, we parametrize the simple $\CX$-truncated Mackey functors for $\CF$, i.e., the simple objects in the abelian category $\Mack^\CX_k(\CF)$, and we describe them explicitly. 

First we consider the case where $\CX$ consists of all subgroups of $S$, following Webb's approach \cite[Section 2]{W1993}. Our constructions here differ from those there only in that, instead of considering all finite groups and homomorphisms among them, we examine only those subgroups and morphisms in the category $\CO(\CF)$. It turns out that this restriction is not essential and the arguments follow verbatim. Hence we omit the proofs, but still give the detailed constructions as they are crucial in the applications.

Let $M\in \Mack_k(\CF)$. A \emph{minimal subgroup} for $M$ is a subgroup $Q\leq S$ such that $M(Q)\neq 0$ but $M(P)=0$ for every proper subgroup $P$ of $Q$.

\begin{prp}[{\cite[Proposition 2.1]{W1993}}]
If $M\in \Mack_k(\CF)$ is simple then $M$ has a unique $\CF$-conjugacy class of minimal subgroups $Q$ and $M(Q)$ is a simple $k\Out_\CF(Q)$-module.
\end{prp}

In the opposite direction, for $Q\leq S$ and $V$ a simple $k\Out_\CF(Q)$-module, we describe a simple Mackey functor $S_{Q,V}\in \Mack_k(\CF)$ for which $Q$ is a minimal subgroup and $S_{Q,V}(Q) \cong V$. 

To define the values of $S_{Q,V}$ we start by setting $S_{Q,V}(L)=V$ as a set for each $\CF$-conjugate $L$ of $Q$. For each such subgroup $L$ we also fix an $\CF$-isomorphism $\alpha\colon Q \to L$. If $\alpha_i\colon Q\to L_i$ $(i=1,2)$ are the chosen isomorphisms for subgroups $L_1$ and $L_2$ that are $\CF$-conjugate to $Q$ and $[\gamma]\colon L_1 \to L_2$ is an isomorphism in $\CO(\CF)$ we define 
\[
{S_{Q,V}}_*([\gamma])(v)=[\alpha_2^{-1}\gamma\alpha_1]\cdot v\text{ for $v\in V$ }
\]
and ${S_{Q,V}}^*([\gamma])={S_{Q,V}}_*([\gamma^{-1}])$. This makes $V=S_{Q,V}(L)$ into the $k\Out_\CF(L)$-module obtained from the $k\Out_\CF(Q)$-module $V$ by transporting the action along $\alpha\colon Q\to L$. We denote this module by ${}^\alpha V$.

For a general subgroup $P\leq S$ we set 
\begin{equation}\label{equ:defSQVonobject}
	S_{Q,V}(P)\cong\bigoplus_{Q\stackrel{\alpha}\cong L\leq_P P} \trm_L^{N_P(L)}({}^\alpha V),
\end{equation}
where the direct sum is taken over the $P$-conjugacy classes of the subgroups $L$ of $P$ that are $\CF$-conjugate to $Q$. For each of these classes a representative $L$ is chosen and an $\CF$-isomorphism $\alpha\colon Q\to L$ is also fixed.  The map $\trm^{N_P(L)}_{L}\colon {}^{\alpha}V \to {}^{\alpha}V$ is the relative trace map, where $N_P(L)$ acts on the $k\Out_\CF(L)$-module ${}^\alpha V$ via the map $N_P(L) \to \Out_\CF(L)$ given by conjugation. 

Now let $R\leq S$ be another subgroup and write 
\[
	S_{Q,V}(R)\cong\bigoplus_{Q\stackrel{\beta}\cong L'\leq_R R} \trm_{L'}^{N_R(L')}({}^\beta V),
\]
where we have chosen representatives $L'$ and isomorphisms $\beta\colon Q\to L'$. If $[\gamma] \colon P \to R$ is an isomorphism in $\CO(\CF)$ then ${S_{Q,V}}_*([\gamma])$ is determined on the summand corresponding to $L\leq P$ via the isomorphism $\gamma|_{L}\colon L\to \gamma(L)\leq R$ in the way described before. We define ${S_{Q,V}}^*([\gamma])={S_{Q,V}}_*([\gamma^{-1}])$.

Suppose now that $R\leq P$. The transfer $\tr^P_R={S_{Q,V}}_*([\iota_R^P])$ sends the summand corresponding to $L'\leq R$ to the summand corresponding to $L=L'\leq P$ as follows, where we assume for simplicity that $\alpha=\beta\colon Q\to L=L'$ are the chosen representatives and isomorphisms:
\begin{equation}\label{equ:defSQVtransfer}
 \begin{array}{cccc}
         &\trm_L^{N_R(L)}({}^\alpha V) &\xrightarrow{\tr^P_R} &\trm_L^{N_P(L)}({}^\alpha V)\\
         & v &\mapsto & \trm_{N_R(L)}^{N_P(L)}(v).
\end{array}
\end{equation}
The restriction $\res^P_R={S_{Q,V}}^*([\iota_R^P])$ sends the summand corresponding to $L\leq P$ to $0$ if no $P$-conjugate of $L$ is contained in $R$. Assume otherwise that there are $P$-conjugates of $L$ lying in $R$ and let $L'_i={}^{p_i} L$ with $p_i\in P$ be the chosen representatives for their $R$-conjugacy classes. Let $\alpha\colon Q\to L\leq P$ be the chosen isomorphism for $L$ and assume  for simplicity that $\beta_i=c_{p_i}\circ \alpha$ are the chosen isomorphisms for the subgroups $L'_i$. Then $\res^P_R$ on the summand for $L$ is the diagonal
\begin{equation}\label{equ:defSQVrestriction}
 \begin{array}{cccc}
         &\trm_L^{N_P(L)}({}^\alpha V) &\xrightarrow{\res^P_R} &\bigoplus_i \trm_{L'_i}^{N_R(L'_i)}({}^{\beta_i}V)\\
         & v &\mapsto & \oplus_i v.
\end{array}
\end{equation}

\begin{prp}[{\cite[Section 2]{W1993}}]\label{proposition:parametrizationofsimples}
Let $\CF$ be a fusion system on a finite $p$-group $S$. The simple Mackey functors in $\Mack_k(\CF)$ are of the form $S_{Q,V}$, where $Q$ runs over the subgroups of $S$ and $V$ runs over the simple $k\Out_\CF(Q)$-modules. Moreover, $S_{Q,V}$ and $S_{R,W}$ are isomorphic in $\Mack_k(\CF)$ if and only if there is an $\CF$-isomorphism $\alpha\colon Q\to R$ such that ${}^{\alpha}V \cong W$ as $k\Out_\CF(R)$-modules.
\end{prp}

Now we turn to the simples in $\Mack^\CX_k(\CF)$ for an arbitrary $\CF$-overconjugacy closed collection $\CX$ of subgroups of $S$. It turns out that there is a straightforward relation between simples in $\Mack_k(\CF)$ and in $\Mack^\CX_k(\CF)$.

\begin{prp}
Let $\CF$ be a fusion system on a finite $p$-group $S$ and let $\CX$ be an $\CF$-overconjugacy closed collection of subgroups of $S$. For $Q\leq S$ and $V$ a simple $k\Out_\CF(Q)$-module, the simple Mackey functor $S_{Q,V}$ for $\CF$ vanishes outside of $\CX$ and restricts to a simple object in $\Mack^\CX_k(\CF)$ if and only if $Q$ belongs to $\CX$, and all simple objects in $\Mack^\CX_k(\CF)$ arise this way.
\end{prp}

\begin{proof}
If $Q \in \CX$, then $S_{Q,V}(P) = 0$ for $P\leq S$ with $P\notin \CX$ because of the construction of $S_{Q,V}$ and because $\CX$ is $\CF$-overconjugacy closed. In this case the restriction $S_{Q,V}|_{\CO(\CX)}$ belongs to $\Mack^\CX_k(\CF)$ because the $\CX$-truncated Mackey decomposition (Definition \ref{D:Mackey-bivariant}(3)) for $S_{Q,V}|_{\CO(\CX)}$ is obtained for free from the Mackey decomposition for $S_{Q,V}$. Conversely, if $S_{Q,V}$ vanishes outside of $\CX$, then $Q$ must belong to $\CX$ because $S_{Q,V}(Q) \cong V \neq 0$. Finally, suppose $M$ is an object of $\Mack^\CX_k(\CF)$ and let $N$ be the ``extension by zero'' of $M$ to $\CO(\CF)$. Then $N\in \Mack_k(\CF)$ and $N|_{\CO(\CX)} = M$. Moreover, $M$ is simple in $\Mack^\CX_k(\CF)$ if and only if $N$ is simple in $\Mack_k(\CF)$.
\end{proof}

We record the following useful property of $S_{Q,V}$.

\begin{cor}[{\cite[Proposition 2.8]{W1993}}] \label{T:non-vanishing} Let $\CF$ be a fusion system on a finite $p$-group $S$. Let $Q$ be a subgroup of $S$ and let $V$ be a simple $k\Out_\CF(Q)$-module and consider the simple Mackey functor $S_{Q,V}$ for $\CF$. If $S_{Q,V}(P) \neq 0$, then there exists a morphism $\alpha\colon Q\to P$ in $\CF$ such that the stabilizer in $N_P(\alpha(Q))$ of ${}^{\alpha}V$ is equal to $\alpha(Q)$; in particular, $C_P(\alpha(Q)) \leq \alpha(Q)$. \label{T:non-vanishing:Stab}
\end{cor}
Recall now that a Mackey functor $M$ is \emph{cohomological} if for every inclusion of subgroups $R\leq P$, restriction followed by transfer is multiplication by the index: 
\[
\tr^P_R (\res_R^P (x))=|P:R|x\text{ for every $x\in M(P)$.}
\]
Th\'evenaz and Webb \cite[Proposition 16.10]{TW1995} showed that a simple Mackey functor for a fixed finite group $G$ over a field of characteristic $p$ is cohomological if and only if its minimal subgroups are $p$-groups. We prove an analogous result for simple Mackey functors for fusion systems. First we recall a group theoretic result.

\begin{lem} \label{T:counting-cosets-mod-p}
Let $G$ be a finite group, $H$ a subgroup of $G$ and $Q$ a $p$-subgroup of $G$. Then
\[
	|N_G(Q,H):H| \equiv |G:H| \mod p,
\]
where $N_G(Q,H) = \{ x\in G \mid {}^{x}Q \leq H \}$.
\end{lem}

\begin{proof}
Consider the right $Q$-action on the set $H\backslash G$ of right $H$-cosets in $G$ induced by multiplication in $G$.  The $Q$-fixed points are $H\backslash N_G(Q,H)$ and the nontrivial $Q$-orbits have cardinality divisible by $p$.
\end{proof}

\begin{prp} \label{T:non-vanishing:cohomological}
Let $\CF$ be a fusion system on a finite $p$-group $S$ and let $k$ be a field of characteristic $p$.  Then every simple Mackey functor for $\CF$ over $k$ is cohomological.
\end{prp}
\begin{proof}
Consider the simple Mackey functor $S_{Q,V}$ in $\Mack_k(\CF)$ with $Q\leq S$ and $V$ a simple $k\Out_\CF(Q)$-module. Let $R\leq P\leq S$ and consider the summand of $S_{Q,V}(P)$ corresponding to $L\leq P$ and for which the $\CF$-isomorphism $\alpha\colon Q\to L$ has been chosen. Let $C$ be the set of $R$-conjugacy classes of $P$-conjugates of $L$ lying in $R$, and choose representatives $L'_i={}^{p_i}L$ for these classes. Using \eqref{equ:defSQVrestriction} and \eqref{equ:defSQVtransfer} it is easy to see that then $\tr^P_R\res^P_R$ restricts to this summand as
\[
\begin{array}{cccc}
         &\trm_L^{N_P(L)}({}^\alpha V) &\xrightarrow{\tr_R^P\res^P_R} &\trm_L^{N_P(L)}({}^\alpha V)\\
         & v &\mapsto & \sum_i \trm^{N_P(L)}_{N_{{}^{q_i}R}(L)}(v),
\end{array}
\]
where $q_i=p_i^{-1}$. Because $v$ already lies in the invariants $({}^\alpha V)^{N_P(L)}$ we have that 
\[
\tr_R^P\res^P_R(v)=\sum_i |N_P(L):N_{{}^{q_i}R}(L)|v.
\]
Now note that $|C|=|R\backslash N_P(L,R)/N_P(L)|$, where $R$ and $N_P(L)$ act on $N_P(L,R)$ by multiplication on the left and right respectively. Moreover, the action of $N_P(L)$ on the orbit of $p_i$ in $R\backslash N_P(L,R)$ has isotropy group equal to $N_{{}^{q_i}R}(L)$. It turns out then that
\[
\tr_R^P\res^P_R(v)=|N_P(L,R):R|v=|P:R|v,
\]
where the last equality is due to Lemma \ref{T:counting-cosets-mod-p}.
\end{proof}

\section{Acyclicity of Mackey functors} \label{S:acyclic-Mackey}

In this section, we present an important acyclicity condition for Mackey functors for fusion systems. Contrary to the previous two sections, saturation (see \cite[Definition 1.2]{BLO2003theory}) is crucial for the following theorem.

\begin{thmA} \label{T:Mackey-acyclic-over-OFc}
Let $\CF$ be a  saturated fusion system on a finite $p$-group $S$.  Then for every $\CF^c$-restricted Mackey functor $M = (M^*,M_*) \colon\CO(\CF^c)\to\Mod{\BZ_{(p)}}$ for $\CF$, the contravariant part $M^*$ is acyclic: $\varprojlim^i_{\CO(\CF^c)} M^*|_{\CO(\CF^c)}=0$ for $i>0$.
\end{thmA}

This theorem is a consequence of the combination of works of Jackowski and McClure \cite{JM1992} and D\'iaz and Libman \cite{DL2009}.  In \cite{JM1992}, Jackowski and McClure,  following Dress's definition of Mackey functors for a finite group $G$, define a \emph{Mackey functor} on a small category $\CC$ with pullbacks as a bivariant functor $M=(M^*,M_*)\colon \CC \to \Mod{\BZ}$ such that every pullback diagram
\[
\xymatrix{
	W \ar[r]^{\gamma}\ar[d]_{\delta} & Y\ar[d]^{\beta}\\
	X \ar[r]^{\alpha} & Z
}
\]
in $\CC$ induces a commutative diagram
\[
\xymatrix{
	M(W) \ar[r]^{M_*(\gamma)}\ar[r] & M(Y)\\
	M(X) \ar[r]^{M_*(\alpha)}\ar[u]^{M^*(\delta)} & M(Z) \ar[u]_{M^*(\beta)}.
}
\]
Further they define a \emph{proto-Mackey functor} on a small category $\CB$ as a a bivariant functor $M\colon\CB\to\Mod{\BZ}$ such that its additive extension $M_\amalg\colon\CB_\amalg\to\Mod{\BZ}$ is a Mackey functor in the previous sense, provided that $\CB_\amalg$ has pullbacks.  Then they show the following acyclicity
criterion.


\begin{prp}[{\cite[Corollary 5.16]{JM1992}}] \label{T:JM-acyclicity}
Let $\CB$ be a small category satisfying the following conditions.
\begin{enumerate}
\labitem{{\rm (B0)}}{B0} The product of each pair of objects of $\CB$ and the pullback of each diagram $c \to e \from d$ of objects of $\CB$ exist in $\CB_\amalg$.
\labitem{{\rm (B1)}}{B1} $\CB$ has finitely many isomorphism classes of objects, each set of morphisms is finite, and all self-maps in $\CB$ are isomorphisms.
\labitem{{\rm (B2)}}{B2} For each object $P$ of $\CB$ there is an object $Q$ of $\CB$ with $|\Hom_\CB(P,Q)|$ prime to~$p$. 
\end{enumerate}
Then for every proto-Mackey functor $M\colon\CB\to\Mod{\BZ_{(p)}}$, its contravariant part $M^*$ is acyclic.
\end{prp}

The condition \ref{B0}, which is denoted as (PB$\times_\amalg$) in \cite{JM1992}, ensures that $\CB_\amalg$ has pullbacks and product of pairs of objects which is distributive over coproducts.  This allows them to define Mackey functors and the Burnside ring for $\CB_\amalg$.  Using the action of the Burnside ring on Mackey functors and the conditions \ref{B1} and \ref{B2}, they prove the above proposition.

On the other hand, D\'iaz and Libman \cite{DL2009} construct the Burnside ring of $\CO(\CF^c)_\amalg$ using the products in $\CO(\CF^c)_\amalg$, which goes back to Puig. It is easy to show that $\CO(\CF^c)_\amalg$ has equalizers.  Consequently, $\CO(\CF^c)_\amalg$ has pullbacks.

\begin{prp}[{\cite[Propositions 2.9 and 2.10]{DL2009}}] \label{T:orbit-centric-extended has pullbacks}
Suppose that $\CF$ is a saturated fusion system on a finite $p$-group $S$.  Then $\CO(\CF^c)$ satisfies the condition \ref{B0}. In particular, for $\CF$-centric subgroups $Q, R\leq P\leq S$, the pullback of $R \xrightarrow{[\iota_{R}^{P}]} P \xleftarrow{[\iota_{Q}^{P}]} Q$ is given by
\[
\xymatrix@C=20mm{
	&\coprod_{x} Q \cap {}^{x}R \ar[r]^{([\iota_{Q \cap {}^{x}R}^{Q}])}\ar[d]_{([c_{x^{-1}}])} &Q\ar[d]^{[\iota_{Q}^{P}]}\\
	&R \ar[r]^{[\iota_{R}^{P}]} &P
}
\]
where the coproduct is taken over $x \in [Q\backslash P/R]_{\CF^c}= \{ x \in [Q\backslash P/R] \mid Q \cap {}^{x}R \in \CF^c \}.$
\end{prp}

\begin{proof}[Proof of Theorem {\hyperref[T:Mackey-acyclic-over-OFc]{A}}]
By Proposition~\ref{T:orbit-centric-extended has pullbacks} and \cite[Lemma 5.13]{JM1992}, proto-Mackey functors on $\CO(\CF^c)$ are in fact the same as $\CF^c$-truncated Mackey functors for $\CF$ as defined in this paper. Thus it suffices to show that $\CO(\CF^c)$ satisfies the three conditions of Proposition~\ref{T:JM-acyclicity}. We have seen that the condition \ref{B0} holds for $\CO(\CF^c)$. The condition \ref{B1} is trivially satisfied and the condition \ref{B2} follows from the well-known fact (\cite[Proposition 2.8]{DL2009}) that $|\Hom_{\CO(\CF^c)}(P,S)|$ is prime to $p$ for every $\CF$-centric $P\leq S$.
\end{proof}

Now we present a strategy to study acyclicity of restrictions of Mackey functors, which will be used in Section \ref{section:Examples.}. Consider a Mackey functor $M = (M^*,M_*) \colon \CO(\CF)\rightarrow \fgmod{\BF_p}$ defined on the full orbit category. Its restriction to the centric orbit category $M|_{\CO(\CF^c)}$ might not be an $\CF^c$-truncated Mackey functor, and so we cannot use Theorem~\hyperref[T:Mackey-acyclic-over-OFc]{A} directly to conclude that $M^*|_{\CO(\CF^c)}$ is acyclic. Instead we break up the Mackey functor $M$ by taking a composition series, then study how the restrictions of the composition factors contribute to the higher limits of $M^*|_{\CO(\CF^c)}$.

\begin{prp}\label{prop:strategy}
Let $\CF$ be a saturated fusion system on a finite $p$-group $S$ and let $k$ be a field of characteristic $p$. Consider a Mackey functor $M\colon\CO(\CF)\to\fgmod{k}$ which takes finite dimensional $k$-vector spaces as its values. If for each composition factor $S_{Q,V}$ of $M$, where $Q\leq S$ is non-$\CF$-centric and $V$ is a simple $k\Out_\CF(Q)$-module, the functor $S_{Q,V}^*|_{\O(\CF^c)}$ is acyclic, then $M^*|_{\O(\CF^c)}$ is acyclic.
\end{prp}
\begin{proof}
By assumption there is a finite-length filtration of $M$ by successive maximal Mackey subfunctors
\[
0=M_r\subset M_{r-1} \subset\ldots\subset M_{1}\subset M_0=M.
\]
Restricting this filtration we get a filtration of bivariant functors $N_i=M_i|_{\CO(\CF^c)}$. From the long exact sequence of higher limits associated to the short exact sequences $0 \to N^*_{i+1} \to N^*_i \to N^*_i/N^*_{i+1} \to 0$, it is immediate that if each $N_i^*/N^*_{i+1}$ is acyclic then so is $N_0^*=M^*|_{\CO(\CF^c)}$. 
Now, by Proposition \ref{proposition:parametrizationofsimples}, each quotient $M_i/M_{i+1}$ is of the form $S_{Q,V}$, where $Q\leq S$ is a subgroup of $S$ and $V$ is a simple $k\Out_\CF(Q)$-module. Hence, each quotient $N_i^*/N^*_{i-1}$ is of the form $S_{Q,V}^*|_{\O(\CF^c)}$. If $Q$ is $\CF$-centric then ${S_{Q,V}}|_{\O(\CF^c)}$ is in $\Mack^{\CF^c}_k(\CF)$ by Proposition \ref{proposition:parametrizationofsimples} and so $S_{Q,V}^*|_{\O(\CF^c)}$ is acyclic by Theorem \hyperref[T:Mackey-acyclic-over-OFc]{A}. Therefore, $M^*|_{\O(\CF^c)}$ is acyclic provided that $S_{Q,V}^*|_{\O(\CF^c)}$ is acyclic whenever $Q$ is not $\CF$-centric.
\end{proof}

The next proposition together with Theorem \hyperref[T:Mackey-acyclic-over-OFc]{A} give sufficient conditions for acyclicity, which may be applied to functors $S_{Q,V}^*|_{\O(\CF^c)}$ with $Q$ non-$\CF$-centric as in the statement of Proposition \ref{prop:strategy}. In its statement and proof we use the notations $\res$, $\tr$ and $\iso$ introduced in Proposition \ref{D:Mackey-centric-res-tr-iso}.

\begin{prp}\label{prop:MackeyFvsMackeyFc}
Let $\CF$ be a saturated fusion system on a finite $p$-group $S$ and let $k$ be a field of characteristic $p$. Consider a Mackey functor $M$ for $\CF$ over $k$. Then the restriction $M|_{\CO(\CF^c)}$ is an $\CF^c$-truncated Mackey functor for $\CF$ if the composite
\[
	M(P) \xrightarrow{\res^P_{P\cap R}} M(P\cap R) \xrightarrow{\tr^R_{P\cap R}} M(R)
\]
is zero whenever $P, R\leq S$ are $\CF$-centric and $P\cap R$ is not $\CF$-centric. 
\end{prp}


\begin{proof}
Since $M$ is a Mackey functor for $\CF$, we have
\[
	\res^{T}_{R} \tr^{T}_{P} = \sum_{x \in [R \backslash T/P]} \tr^{R}_{R\cap{}^{x}P} \res^{{}^{x}P}_{R\cap {}^{x}P} \iso(c_x|_P)
\] 
for all $P, R\leq T\leq S$.  As remarked at the end of Section~\ref{section:Mackey functors.}, the restriction $M|_{\CO(\CF^c)}$ is an $\CF^c$-truncated Mackey functor for $\CF$ if the above equality holds with the sum replaced by the sum over $x \in [P \backslash T/R]_{\CF^c}$ whenever $P$, $R$, $T$ are $\CF$-centric.  In particular, this is the case provided that 
\[
	\tr^{R}_{R\cap{}^{x}P} \res^{{}^{x}P}_{R\cap {}^{x}P} =0
\]
whenever $P\cap {}^xR$ is not $\CF$-centric. 
The proposition follows.
\end{proof}

\begin{exm}
Assume the notations of Proposition \ref{prop:MackeyFvsMackeyFc} with $M = S_{Q,V}$ for some non-$\CF$-centric subgroup $Q$ of $S$ and a simple $k\Out_\CF(Q)$-module $V$. In general $\tr^R_{P\cap R}$ is nonzero: For instance, for $p=3$ and $m\geq 2$, let $s_1$ and $s_2$ be the generators of the abelian $3$-group $\gamma_1=\BZ_{3^m}\times \BZ_{3^m}$ and consider  $G=\gamma_1\rtimes \SL_2(3)$. A Sylow $3$-subgroup of $G$ is $S=\gamma_1\rtimes \langle s \rangle$, where $s$ has order $3$ and acts via the matrix $\bigl(\begin{smallmatrix}
1&-3\\ 1&-2
\end{smallmatrix} \bigr)$. The fusion system $\CF=\CF_S(G)$ corresponds to the fourth line of \cite[Table 6]{DRV2007} and $S$ is the maximal nilpotency class $3$-rank two $3$-group $B(3,2m+1;0,0,0)$. 

Now, the subgroup $Q=\langle s_1^{3^{m-1}},s_2^{3^{m-1}}\rangle$ is isomorphic to $\BZ_3\times \BZ_3$ and is not centric as $C_S(Q)=\gamma_1$. The element $s$ acts on $Q$ via the matrix $\bigl(\begin{smallmatrix}
1&0\\ 1&1
\end{smallmatrix} \bigr)$ and hence the subgroup $P=Q\rtimes \langle s\rangle$ is isomorphic to the extraspecial group of order $27$ and exponent $3$, $3^{1+2}_+$. Moreover, $P$ is $\CF$-centric. Because $\Out_\CF(\gamma_1)=\SL_2(3)$ and this group is $3$-reduced, this outer automorphism group restricts to $\Out_\CF(Q)=\SL_2(3)$.

Let $V$ be the unique $3$-dimensional simple $k\SL_2(3)$-module. Then the action of $\overline s\in \Out_\CF(Q)$ on $V$ can be represented by $T=\Bigl(\begin{smallmatrix}
1&1&1\\ 0&1&-1\\0&0&1
\end{smallmatrix} \Bigr)$. Since $I+T+T^2=\Bigl(\begin{smallmatrix}
0&0&1\\ 0&0&0\\0&0&0
\end{smallmatrix} \Bigr)$, we have $\trm_Q^P(V)=k$ and the transfer map
\[
S_{Q,V}(Q)=V\xrightarrow{\tr^P_{Q}} S_{Q,V}(P)=\trm_Q^P(V)=k
\]
is nonzero (cf. \ref{equ:defSQVtransfer}).


In the proof of Theorem {\hyperlink{theorem:index-p-case}{C}} in the next section, we will nevertheless show that under some assumptions (including the above situation) the composition in Proposition \ref{prop:MackeyFvsMackeyFc} with $M = S_{Q,V}$ and $Q$ a non-$\CF$-centric subgroup is zero.
\end{exm}

\section{Examples.}\label{section:Examples.}

In this section we provide confirmation of the sharpness conjecture in the introduction for the finite group case and for those fusion systems defined on a $p$-group having an abelian subgroup of index $p$. First we deal with the group case. 
For the definition and properties tailored to fusion systems, see \cite[III.\S5.4]{AKO2011}.


\begin{thmB} \label{theorem:group-case}
Let $\CF=\CF_S(G)$ be the fusion system of the finite group $G$ with $S\in \Syl_p(G)$ and let $M=(M^*,M_*)\colon \CO(\CF)\to \fgmod{\BF_p}$ be a Mackey functor over the full orbit category of $\CF$. Then the higher limits of $M^*$ over the centric orbit category vanish: $\varprojlim^i_{\CO(\CF^c)} M^*|_{\CO(\CF^c)}=0$ for $i>0$.
\end{thmB}

\begin{proof}
Consider the following commutative diagram:
\[
\xymatrix{
\CO_p(G)\ar[r]^>>>>>>{\Psi}&\CO_S(G)\ar[r]^>>>>>{\Phi}&\CO(\CF)\ar[r]^{M^*}&\fgmod{\BF_p}\\
\CO^c_p(G)\ar@{^(->}[u]\ar[r]^>>>>>>{\Psi}&\CO^c_S(G)\ar@{^(->}[u]\ar[r]^>>>>>{\Phi}&\CO(\CF^c)\ar@{^(->}[u]\ar[ru]_{M^*|_{\CO(\CF^c)}}\\
}
\]
Here $\CO_S(G)$, $\CO^c_p(G)$ and $\CO^c_S(G)$ are full subcategories of the orbit category $\CO_p(G)$ with objects the $p$-subgroups of $G$ and morphisms given by $\Hom_{\CO_p(G)}(P,Q)=Q \backslash N_G(P,Q)$. The objects of $\CO_S(G)$, $\CO^c_p(G)$ and $\CO^c_S(G)$ are, respectively, the subgroups of $S$, the $p$-centric subgroups of $G$, and the $p$-centric subgroups of $G$ contained in $S$. Recall that a $p$-subgroup $P$ of $G$ is said to be $p$-centric if $Z(P)\in \Syl_p(C_G(P))$, or equivalently, if $C_G(P) = Z(P) \times O_{p'}(C_G(P))$. If $P\leq S$, then $P$ is $p$-centric if and only if $P$ is $\CF$-centric. The functor $\Phi\colon \CO_S(G) \to \CO(\CF)$ is the canonical projection functor. The inclusion $\CO_S(G) \into \CO_p(G)$ is an equivalence of categories and $\Psi\colon \CO_p(G)\to \CO_S(G)$ denotes an inverse equivalence. The functors $\Phi$ and $\Psi$ restrict, respectively, to the canonical projection and an equivalence of categories between the full subcategories of $p$-centric subgroups.

Note that $\CO_S(G)$ satisfies all conditions in Proposition \ref{T:JM-acyclicity}, because $\CO_p(G)$ does and $\Psi$ is an equivalence of categories. The composite $M\Phi = (M^*\Phi, M_*\Phi) \colon \CO_S(G) \to \fgmod{\BF_p}$ is a proto-Mackey functor for the category $\CO_S(G)$. By Proposition \ref{T:JM-acyclicity}, we conclude that the contravariant part $M^*\Phi$ is acyclic. Since an equivalence of categories preserves higher limits (see for example \cite[III.5.6]{AKO2011}), it follows that $M^*\Phi\Psi$ is acyclic.

Now, by \cite[Theorem 1.2]{G2002}, we can discard the non $p$-centric subgroups when evaluating higher limits, i.e., 
\[
	{\varprojlim}^*_{\CO_p(G)} M^*\Phi\Psi\cong {\varprojlim}^*_{\CO^c_p(G)} M^*\Phi\Psi.
\]
Finally, since $\Phi\Psi|_{\CO^c_p(G)}$ is the quotient by $p'$-group (see \cite[Lemma
1.3]{{BLO2003equiv}}), we have
\[
{\varprojlim}^*_{\CO^c_p(G)} M^*\Phi\Psi \cong {\varprojlim}^*_{\CO(\CF^c)} M^*,
\]
which completes the proof.
\end{proof}

Next we shall apply the method described in Proposition \ref{prop:strategy} together with Theorem \hyperref[T:Mackey-acyclic-over-OFc]{A} and Proposition \ref{prop:MackeyFvsMackeyFc} to confirm the sharpness conjecture for a special class of fusion systems, including some exotic fusion systems.

\hypertarget{theorem:index-p-case}{\begin{thmC}} 
Let $S$ be a nonabelian finite $p$-group with an abelian subgroup of index $p$ and let $\CF$ be a saturated fusion system $\CF$ on $S$. For any Mackey functor $M=(M^*,M_*)\colon \CO(\CF)\rightarrow \fgmod{\BF_p}$, we have 
\[
{{\varprojlim}^i_{\CO(\CF^c)}} M^*|_{\CO(\CF^c)}= 0\qquad\forall i>0.
\]
\end{thmC}


The structure of the group $S$ and all possible $\CF$-essential subgroups of $S$ are analyzed in \cite{Oindp1}, which we summarize in the following lemmas.  First we distinguish two cases.

\begin{lem}[{\cite[Lemma 1.9]{Oindp1}}] \label{T:index-p-S}
Let $S$ be a nonabelian finite $p$-group with an abelian subgroup $A$ of index $p$. Then either
\begin{enumerate}
\item $|S/Z(S)| \geq p^3$ and $A$ is the unique abelian subgroup of $S$ of index $p$, or
\item $S/Z(S) \cong C^2_p$ and $S$ has exactly $p+1$ abelian subgroups of index $p$,
\end{enumerate}
\end{lem}

In the second case, the $p+1$ abelian subgroups of $S$ of index $p$ are precisely the proper centric subgroups of $S$ (i.e., subgroups $P<S$ such that $C_S(P)\leq P$), and $\CF$-essential subgroups occur only among those. The first case is dealt with in the following lemma.

\begin{lem}[{\cite[Lemma 2.3]{Oindp1}}] \label{T:index-p-essential}
Let $S$ be a nonabelian finite $p$-group with a unique abelian subgroup $A$ of index $p$.  Set $Z=Z(S)$, $Z_2 = Z_2(S)$.  Let $\CF$ be a saturated fusion system on $S$.  
\begin{enumerate}
\item Suppose $P< S$ is an abelian $\CF$-essential subgroup such that $P\neq A$. Then $P\cap A = Z$, $N_A(P) = Z_2$ and $|N_S(P)/P| = |Z_2/Z|= p$.
\item Suppose $P< S$ is a nonabelian $\CF$-essential subgroup.  Then $Z(P) = Z$, $P\cap A = Z_2$, $|N_S(P)/P| = |Z_2/Z|= p$.
\end{enumerate}
\end{lem}

To apply our method, we also need to know the centric subgroups of $S$.

\begin{lem} \label{T:index-p-centric} 
Let $S$ be a nonabelian finite $p$-group with an abelian subgroup $A$ of index $p$.  Set $Z=Z(S)$ and let $P< S$.
\begin{enumerate}
\item If $P\leq A$, then $P$ is centric in $S$ if and only if $P=A$.
\item If $P\not\leq A$, then $P$ is centric in $S$ if and only if $Z\leq P$.  More precisely, if $P\not\leq A$ is abelian centric, then $P\cap A = Z$; if $P\not\leq A$ is nonabelian centric, then $Z(P) = Z$.
\end{enumerate}
\end{lem}

\begin{proof}
The first part is obvious.  So assume that $P\not\leq A$. If $P$ is centric, then clearly $Z\leq P$.  Conversely, suppose that $Z\leq P$.  Take $x\in P - A$.  Then $P = P_0\la x \ra$ where $P_0 = P\cap A$ and $S = A\la x \ra$.  Now $C_S(P) \leq C_S(x)$ and $C_S(x)$ has $C_A(x) = Z$ as an index $p$ subgroup with $x \in C_S(x) - Z$.  Thus $C_S(P) \leq C_S(x) = Z\la x\ra \leq P$ and so $P$ is centric. If $P$ is abelian, then $S = AP$ and both $A$ and $P$ are abelian, so $P_0\leq Z\leq P$.  Since $|P:P_0|=p$ and $P\not\leq A$, it follows that $P_0=Z$.  If $P$ is nonabelian, then $Z \leq Z(P) = C_S(P) \leq Z\la x \ra$, so either $Z(P) = Z$ or $Z(P) = Z\la x\ra$.  But if $Z(P) = Z\la x\ra$, then $P = P_0Z(P)$ and so $P$ is abelian, a contradiction.  Thus $Z(P) = Z$.
\end{proof} 

The above information will reduce the configurations we need to consider to simple situations which can be handled easily.

\begin{proof}[Proof of Theorem {\hyperlink{theorem:index-p-case}{C}}]
Consider the simple Mackey functor $S_{Q,V}$ with non-$\CF$-centric $Q\leq S$. We will show that $S_{Q,V}|_{\CO(\CF^c)}$ is an $\CF^c$-restricted Mackey functor for $\CF$, and hence that it is  acyclic by Theorem \hyperref[T:Mackey-acyclic-over-OFc]{A}. This will prove the theorem by Proposition \ref{prop:strategy}. In fact, by Proposition \ref{prop:MackeyFvsMackeyFc}, it suffices to show that the composite
\[
	S_{Q,V}(P) \xrightarrow{\res^P_{P\cap R}} S_{Q,V}(P\cap R) \xrightarrow{\tr^R_{P\cap R}} S_{Q,V}(R)
\]
is zero for any two $\CF$-centric subgroups $P, R\leq S$ such that $T:=P\cap R$ is not $\CF$-centric. So let $P$ and $R$ be such subgroups and note then that $Z:=Z(S) \leq C_S(P)\cap C_S(R) \leq P\cap R$.  Also we may assume that 
\begin{enumerate}
\item[(a)] $P$, $R$ are nonabelian proper subgroups of $S$, and that
\item[(b)] $Q$ is $\CF$-conjugate to $T$ if $T$ is abelian.
\end{enumerate}
Indeed, both $P$ and $R$ are proper subgroups of $S$ because otherwise $P\cap R$ is $\CF$-centric.  Since $P$, $R$ are $\CF$-centric while $T=P\cap R$ is not, we see that $T$ is properly contained in both $P$ and $R$.  On the other hand, if $U\leq S$ and $S_{Q,V}(U)\neq 0$, then there is $L\leq U$ which is $\CF$-conjugate to $Q$ and such that $C_U(L)\leq L$ by Corollary \ref{T:non-vanishing}. Thus if, moreover, $U$ is abelian, then $U = L$ and so $U$ is $\CF$-conjugate to $Q$. This proves both (a) and (b).

If $S$ has more than one abelian subgroup of index $p$, then the only nonabelian centric subgroup of $S$ is $S$ itself by the comments after Lemma \ref{T:index-p-S}. So we may assume that $A$ is the unique abelian subgroup of $S$ of index $p$.

Now first suppose that $T\leq A$.  Then $T$ is abelian and so we may assume that there is an $\CF$-isomorphism $\alpha\colon Q\to T$ by the above observation (b).  Also we may assume that $P\cap A = T = R\cap A$. Indeed, if $P\cap A > T$, then since $P\cap A$ is abelian and $|P\cap A|>|Q|$, we have $S_{Q,V}(P\cap A)=0$ by Corollary \ref{T:non-vanishing}, and so the composite is zero (and similarly for $R\cap A>T$). Now $T$ is normal in both $P$ and $R$ as $A$ is normal in $S$, so the composite reduces to
\[
	\trm^P_T({}^{\alpha}V) \xrightarrow{\res^P_T} {}^{\alpha}V \xrightarrow{\tr^R_T} \trm^R_T({}^{\alpha}V).
\]
The homomorphisms from $P/T$ and $R/T$ to $S/A\cong C_p$ (induced by the inclusions $P \hookrightarrow S$ and $R \hookrightarrow S$) are both isomorphisms. If $C_P(T)=P$ the transfer $\trm^P_T$ and the composite above are both zero (and similarly for $R$). So we may assume that $C_P(T)=C_R(T)=T$ and hence that $C_S(T)=A$ too. Thus, both $\Out_P(T)\cong P/T$ and $\Out_R(T)\cong R/T$ are equal to $\Out_S(T)\cong S/A$. Therefore we have $\trm^R_T({}^{\alpha}V) = \trm^P_T({}^{\alpha}V)$ and the composite in the diagram equals the composite $\tr^P_T\res^P_T$, and this is zero by Proposition \ref{T:non-vanishing:cohomological}.

Now suppose $T\not\leq A$. By an observation at the beginning of the proof we have $Z\leq T$ and hence,  by Lemma \ref{T:index-p-centric},  $T$ is centric in $S$. Note that, in fact, as $Z\leq A$, we must  have $Z<T$. If there is no $\CF$-essential subgroups of $S$ properly containing $T$, then the $\CF$-conjugates of $T$ (other than $T$ itself) are given by the $\CF$-automorphisms of $S$ by Alperin's fusion theorem, and so they also contain $Z$ and are not contained in $A$.  So all $\CF$-conjugates of $T$ are centric, whence $T$ is $\CF$-centric, contradicting the assumption.  Thus there is at least one $\CF$-essential subgroup $U$ of $S$ properly containing $T$. By Lemma~\ref{T:index-p-essential}, either $U$ is abelian and $|U/Z|=p$ or $U$ is nonabelian and $|U/Z|=p^2$.  Since $Z<T<U$, it follows that $U$ is nonabelian, whence $U=Z_2T$ and $|Z_2/Z|=p$ again by Lemma~\ref{T:index-p-essential}.

We show that both $P$ and $R$ contain $Z_2$.  Then $P$ and $R$ contain $U$ as they already contain $T$. Then the theorem follows because the composite in question factors through $\tr^U_T \res^U_T$, which is zero by Proposition \ref{T:non-vanishing:cohomological}. Since $Z(S/Z)=Z_2/Z$ has order $p$, it is contained in every nontrivial normal subgroup of $S/Z$.  Since $P\not\leq A$, we have $S=AP$ and so $P_0=P\cap A$ is a normal subgroup of $S$. But $P_0$ properly contains $Z$ since otherwise $P$ is abelian.  Thus $Z_2 \leq P$ and similarly $Z_2\leq R$, as desired. 
\end{proof}

The following example shows that the restriction to $\CO(\CF^c)$ of a Mackey functor
for $\CF$ may not be an $\CF^c$-restricted Mackey functor for $\CF$. So to prove Conjecture in the introduction it is not enough to employ Theorem \hyperref[T:Mackey-acyclic-over-OFc]{A} and certain strategy, like that in the proof of Theorem {\hyperlink{theorem:index-p-case}{C}}
, must be used.
\begin{exm}[D\'iaz--Libman]\label{example:nonmackey}
We will exhibit a finite $p$-group $S$ such that the degree $1$ cohomology functor $H^1(-;\BF_p)\colon \CO(\CF^c)\rightarrow \fgmod{\BF_p}$ is not an $\CF^c$-restricted Mackey functor for the nilpotent fusion system $\CF=\CF_S(S)$. By the sharpness result Theorem B we know that ${\varprojlim}^i_{\CO(\CF^c)} H^1(-;\BF_p)=0$ for $i\geq 1$ though. 

The $p$-group $S$ will have order $p^{p+3}$ ($p$ any prime) and two centric normal subgroups $P$, $Q\unlhd S$ such that:
\begin{enumerate}[\rm (a)]
\item $PQ=S$, $P\cap Q$ is not centric.
\item $H^1(Q;\BF_p)\stackrel{\tr}\rightarrow H^1(S;\BF_p) \stackrel{\res}\rightarrow H^1(P;\BF_p)$ is nonzero.
\end{enumerate}
This proves that $H^1(-;\BF_p)\colon \CO(\CF_S(S)^c)\rightarrow \fgmod{\BF_p}$ is not a Mackey functor.

Consider first the Jordan block matrix $A_n$ of size $n$ and eigenvalue $1$. It is easy to check that, for a prime $p$, $A_p$ has order $p$ working in $\BF_p$. Now define the $(p+2)\times(p+2)$ matrix with two Jordan blocks, one of size $p$ and one of size $2$:
\[
 B_p = 
 \begin{pmatrix}
  A_p & 0 & 0\\
  0 & 1 & 1 \\
  0 & 0 & 1 \\
   \end{pmatrix}.
\]
Then $B_p$ also has order $p$ working in $\BF_p$, and the eigenspace associated to eigenvalue $1$ consists of all vectors $ \begin{pmatrix} *&  0&  \hdots&  0&  *&  0\end{pmatrix}^t$, where ${}^t$ stands for transposition.

Define $Q=\BZ_p\times \BZ_p \times \cdots \times \BZ_p$ ($(p+2)$ copies) and $S=Q\rtimes \langle B_p \rangle$. Notice that the subspace $U$ of $Q$ with last coordinate equal to zero is invariant under $B_p$. Define $P=U \rtimes \langle B_p \rangle$. It is clear than $PQ=S$ and that $P\cap Q=U$ is not centric as $Q\leq C_S(U)$. The subgroup $Q$ is centric, and the subgroup $P$ is centric because $C_S(P)\leq C_S(\langle B_p \rangle)\leq P$. So we are left to prove that transfer from $Q$ followed by restriction to $P$ is nonzero. 

Recall that the transfer $\tr:H^1(H;\BF_p)\rightarrow H^1(G;\BF_p)$, where $H\leq G$, is given by
\[
\tr(f)(x)=\sum_i f(t^{-1}_{\sigma_i} x t_i),
\]
where $f\colon H\to \BF_p$ is any homomorphism, $\{t_i\}$ is a set of representatives for the cosets $G/H$ and $\sigma$ is the permutation of the representatives induced by multiplication by $x$ on the cosets $G/H$. 

Our situation is $H^1(Q;\BF_p)\stackrel{\tr}\rightarrow H^1(Q\rtimes \langle B_p\rangle;\BF_p) \stackrel{\res}\rightarrow H^1(U\rtimes \langle B_p\rangle;\BF_p)$. Consider an arbitrary homomorphism $f:Q=\BZ_p\times \cdots\times\BZ_p\rightarrow \BF_p$ and the element $u\in U\leq U\rtimes \langle B_p \rangle$ given by $u= \begin{pmatrix}0&0&\hdots&0&1&0&0\end{pmatrix}^t$. Then $\res(\tr(f))(u)=\sum^{p-1}_{i=0} f(t^{-1}_i u t_i)$ as the permutation induced by $u$ on the cosets is the identity (as $Q$ is normal in $S$ and $u\in Q$). Taking $I,B_p,\cdots,B^{p-1}_p$ as representatives we obtain 
\[
	\res(\tr(f))(u)=\sum^{p-1}_{i=0} f((B_p)^i(u))=\sum^{p-1}_{i=0} f((A_p)^i(v))=f((I+A_p+\cdots +A^{p-1}_p)v),
\]
where $v={\begin{pmatrix}0&0&\hdots&0&1\end{pmatrix}}^t$. Hence we have to consider the sum $C=I+A_p+\cdots +A^{p-1}_p$. Write $A_p=I+N$, where $I$ is the identity matrix and $N$ is a nilpotent matrix. Then $C=\sum_{i=0}^{p-1} (I+N)^i=\sum_{i=0}^{p-1} \sum_{j=0}^i {i \choose j} N^j=\sum_{j=0}^{p-1}(\sum_{i=j}^{p-1}{i\choose j}) N^j$.
The identity $\sum_{i=j}^{p-1}{i\choose j}={p\choose j+1}$ implies that over $\BZ_p$ the matrix $C$ has a unique non-zero entry equal to $1$  in the upper right corner. Hence we deduce that
\[
	\res(\tr(f))(u)=f(Cv)=f(C\begin{pmatrix}0&0&\hdots&0&1 \end{pmatrix}^t)=f(\begin{pmatrix}1&0&\hdots&0&0 \end{pmatrix}^t).
\]
Choosing $f:\BZ_p\times\hdots\times \BZ_p\rightarrow \BF_p$ which does not vanish on the first coordinate we are done.
\end{exm}


\begin{thebibliography}{1}

\bibitem{AKO2011} M.~Aschbacher, R.~Kessar and B.~Oliver, Fusion Systems in Algebra and Topology, \emph{London Mathematical Society Lecture Note Series}, Vol. \textbf{391}, (2011).

\bibitem{BK1972} A.K.~Bousfield\ and D.~Kan, \emph{Homotopy limits, completions and localizations}, LNM \textbf{304}, Springer-Verlag, Berlin, (1972).

\bibitem{BLO2003theory}
C.~Broto, R.~Levi and B.~Oliver, The homotopy theory of fusion systems, \emph{J. Amer. Math. Soc.} \textbf{16} (2003), no.~4, 779--856
  (electronic).

\bibitem{BLO2003equiv}
C.~Broto, R.~Levi and B.~Oliver, Homotopy equivalences of
  {$p$}-completed classifying spaces of finite groups, \emph{Invent. Math.}   \textbf{151} (2003), no.~3, 611--664.
    
\bibitem{B1998} S.~Bouc, R\'esolutions de foncteurs de Mackey, Group representations: cohomology, group actions and topology (Seattle, WA, 1996), 31--83, \emph{Proc. Sympos. Pure Math.} \textbf{63}, Amer. Math. Soc., Providence, RI, (1998). 

\bibitem{B2010} S.~Bouc, \emph{Biset functors for finite groups}, Lecture Notes in Mathematics, 1990. Springer-Verlag, Berlin, 2010.

\bibitem{BST2012}
S.~Bouc, R.~Stancu and J.~Th\'evenaz, Simple biset functors and double Burnside ring, \emph{J. Pure Appl. Algebra} \textbf{217} (2013), no. 3, 546-566. 

\bibitem{C2011} 
A.~Chermak, Fusion systems and localities, \emph{Acta Math.} \textbf{211} (2013), no.~1, 47--139.

\bibitem{DL2009} A.~D\'iaz and A.~Libman, The Burnside ring of fusion systems, \emph{Advances in Mathematics} \textbf{222} (2009), Issue 6, 1943-1963.

%

\bibitem{DRV2007}
A.~D{\'{\i}}az, A.~Ruiz and A.~Viruel, All $p$-local finite groups of rank two for odd prime $p$, \emph{Trans. Amer. Math. Soc.} \textbf{359} (2007), no. 4, 1725--1764 (electronic).

\bibitem{D1998} W.G.~Dwyer, Sharp homology decompositions for classifying spaces of finite groups, Group representations: cohomology, group actions and topology (Seatle, WA, 1996), \emph{Amer. Math. Soc.}, Providence, RI, (1998), 197--220


\bibitem{G2002} J.~Grodal, Higher limits via subgroup complexes, \emph{Ann. of Math. (2)}, Vol. \textbf{155} (2002), no. 2, 405--457.
 

\bibitem{JM1992} S.~Jackowski and J.~McClure, Homotopy decomposition of classifying spaces via elementary abelian subgroups, \emph{Topology} \textbf{31} (1992), no. 1, 113--132.

\bibitem{JMO1992} S.~Jackowski, J.~McClure and B.~Oliver, Homotopy classification of self-maps of BG via G-actions. II., \emph{Ann. of Math. (2)} \textbf{135} (1992), no. 2, 227--270. 


\bibitem{Oindp1} B.~Oliver, Simple fusion systems over {$p$}-groups with abelian subgroup of index {$p$}: {I}, \emph{J. Algebra} \textbf{398} (2014), 527--541.
  
\bibitem{RV2004} A.~Ruiz and A.~Viruel, The classification of p-local finite groups over the extraspecial group of order $p^3$ and
exponent $p$, \emph{Math. Z.} \textbf{248} (2004), no.~1, 45--65.

\bibitem{TW1995}
J.~Th{\'e}venaz and P.~Webb, The structure of {M}ackey functors, \emph{Trans. Amer. Math. Soc.} \textbf{347} (1995), no.~6, 1865--1961.

\bibitem{W1993}
P.~Webb, Two classifications of simple {M}ackey functors with applications to group cohomology and the decomposition of classifying spaces, \emph{J. Pure Appl. Algebra} \textbf{88} (1993), no.~1-3, 265--304.


\end{thebibliography}
\end{document}